\documentclass{article}

\usepackage[utf8]{inputenc} 
\usepackage[T1]{fontenc}    
\usepackage{hyperref}       
\usepackage{url}            
\usepackage{lipsum}
\usepackage{fancyhdr}       
\usepackage{graphicx}       
\usepackage{amssymb}
\usepackage{tikz}
\usetikzlibrary{decorations.pathreplacing,calligraphy}
\usepackage{caption}
\usepackage{subcaption}
\usepackage{hyperref}
\usepackage{amsthm}

\newtheorem{thm}{Theorem}

\theoremstyle{definition}
\newtheorem{rmk}{Remark}

\usepackage{amsmath} 

\newcommand{\dist}{\textup{dist}}
\newcommand{\diag}{\textup{diag}}

\newtheorem{definition}{Definition}

\usepackage{algpseudocode}
\usepackage{algorithm}
\usepackage{here}

\newcommand{\R}{\mathbb R}
\newcommand{\N}{\mathbb N}
\newcommand{\bfzero}{\mathbf{0}}
\newcommand{\dup}{\textup{d}}
\newcommand{\ddx}[1]{\frac{\dup #1}{\dup x}}

\newcommand{\G}{\mathcal{G}}
\newcommand{\V}{\mathcal{V}}
\newcommand{\E}{\mathcal{E}}
\newcommand{\sumedges}{\sum_{e \in \mathcal{E}}}
\newcommand{\sumedgesincident}{\sum_{e \in \mathcal{E}_v}}
\newcommand{\ellvec}{\boldsymbol{\ell}}
\newcommand{\elle}{\ell_e}
\newcommand{\Ham}{\mathcal{H}}
\newcommand{\Hamu}{\Ham: u \mapsto - \frac{\dup ^2u}{\dup x^2}}
\newcommand{\domHam}{\textup{dom}_{\Ham}}
\newcommand{\domHamNK}{\textup{dom}_{\Ham,\textup{NK}}}

\newcommand{\Ad}{\mathbf{A}}
\newcommand{\Lapl}{\mathbf{L}}
\newcommand{\Deg}{\mathbf{D}}

\newcommand{\normLapl}{\mathcal{L}}

\newcommand{\harmLaplG}{\Delta_{\G}}

\newcommand{\eigvaldiscrete}{\mu}
\newcommand{\eigvec}{\Phi}
\newcommand{\specquantum}{\sigma(\Gamma)}
\newcommand{\NVE}{\left(\frac{k \pi}{\ell} \right)^2}
\newcommand{\NVEe}{\left(\frac{k \pi}{\elle} \right)^2}
\newcommand{\eigvalquantum}{\lambda}
\newcommand{\eigvalquantumsqrt}{\sqrt{\lambda}}
\newcommand{\eigfunc}{\phi}

\newcommand{\eigfunce}{\phi_{e}}
\newcommand{\bfH}{\mathbf{H}}

\newcommand{\NEPvar}{z}
\newcommand{\newtonvar}{j}

\newcommand{\trace}{\textup{trace}}

\usepackage{accents}

\newcommand{\Gammaext}{\tilde{\Gamma}}
\newcommand{\cleangraph}[1]{\textup{clean}(#1)}
\newcommand{\Gammaapprox}{\mathfrak{G}}
\newcommand{\ellapprox}{\mathfrak{l}}

\newcommand{\equilength}{h}
\newcommand{\noequisubdiv}{N}

\newcommand{\Gammaapproxfloorh}[1]{\mathfrak{G}_{\text{fl},#1}}
\newcommand{\Gammaapproxceilh}[1]{\mathfrak{G}_{\text{ce},#1}}

\title{Numerical Computation of Non-Equilateral Quantum Graph Spectra\footnote{This work was supported by HYPATIA.SCIENCE, an initiative for the promotion of young female scientists at the Department of Mathematics and Computer Science of the University of Cologne.}}

\author{
  Chong-Son Dr\"oge, Anna Weller \\
  Department of Mathematics and Computer Science 
  \\
  Faculty of Mathematics and Natural Sciences \\
  University of Cologne, Cologne\\
  \texttt{cdroege@uni-koeln.de, anna.weller@uni-koeln.de} \\
}

\begin{document}
\maketitle

\setlength\parindent{0pt}
\begin{abstract}
In the broad range of studies related to quantum graphs, quantum graph spectra appear as a topic of special interest. 
They are important in the context of diffusion type problems posed on metric graphs. 
Theoretical findings suggest that quantum graph eigenvalues can be found as the solutions of a nonlinear eigenvalue problem, and in the special case of equilateral graphs, even as the solutions of a linear eigenvalue problem on the underlying combinatorial graph.
The latter, remarkable relation to combinatorial graph spectra will be exploited to derive a solver for the general, non-equilateral case.
Eigenvalue estimates from equilateral approximations will be applied as initial guesses in a Newton-trace iteration to solve the nonlinear eigenvalue problem.
\end{abstract}



\section{Introduction}\label{sec:intro}
The application of quantum graphs in physics and engineering dates back to the 1930s (\cite{Pauling_1936}, \cite{Berkolaiko_Kuchment_2013}). 
For example, they appear in the modeling of diffusion type problems on thin, interconnected structures.
These structures can be described in a simplified form as graphs interpreted as topological space, so-called metric graphs. Equipped with a differential operator and vertex coupling conditions, we derive the concept of a quantum graph.
\\
Despite their popularity, the numerical study of related problems is still in its early stages.  
In the context of diffusion problems, the eigenvalues and eigenfunctions of the negative second order derivative play an important role in studying  dynamics on the graph. More precisely, we are interested in the negative second order derivative acting on functions satisfying Neumann-Kirchhoff vertex coupling conditions. The study of these eigenvalues is well-known under the term quantum graph spectrum. The particular difficulty lies in the interconnected structure of (metric) graphs, which is reflected in the quantum graph spectrum.
\\

There is a large body of literature examining the spectrum of quantum graphs from a theoretical point of view. 
In particular, the special situation where all edges in the graph have the same length has been discussed in an early work by van Below \cite{Below_1984}, developing a connection between quantum and combinatorial graph spectra that has also been discussed by several other authors. 
One variant of the relation reduces the spectrum of a quantum graph to the spectrum of the harmonic graph Laplacian matrix of the underlying combinatorial graph \cite{Berkolaiko_Kuchment_2013}. 
This allows, up to some exceptional cases, for computing the quantum graph eigenvalues as well as the eigenfunctions using the solutions of a linear eigenvalue problem.   
\\
In general, we will see that we can relate the spectrum to a so-called nonlinear eigenvalue problem (NEP), the dimension of which is given by the number of vertices of the graph, denoted by $n$.
The solutions of the NEP can be found by applying, for example, a Newton-trace method with suitable initial guesses.
The main objective of this article is to determine eigenvalue estimates to start the Newton iteration by approximating non-equilateral graphs with equilateral extended graphs.
\\

In fact, it is also possible to reduce the continuous eigenvalue problem to a problem of size $2m \times 2m$ where $m$ is the number of edges of the graph. 
In this context, the numerical solution of a secular equation related to a bond scattering matrix has been studied using a spectral counting function \cite{Schanz_2006}.
In \cite{Brio_2022}, the authors use yet another formulation of the NEP.
For the solution, they apply a line minimization algorithm to intervals that are detected by plotting the reciprocal condition number of the matrix-valued function and graphically estimating the minimum spacing between the roots. 
\\
However, since typically $m \gg n$, the dimension of these problems exceeds that of the NEP consulted in this article. Moreover, working with equilateral approximations, we wanted to pursue an approach that takes advantage of the topological structure of the metric graph as well as the reduction to a linear eigenvalue problem in the equilateral case.
\\

The structure of the article is as follows. 
Metric and quantum graphs as well as the notation used in connection with classical combinatorial graphs are introduced in Section \ref{sec:background}.
The formal definition of quantum graph spectra and their relation to nonlinear eigenvalue problems are outlined in Section \ref{sec:quantum_graph_spectra}. 
In Section \ref{sec:equil_approximations}, we further reduce the NEP to a linear eigenvalue problem in the special case of equilateral graphs and describe the approximation of non-equilateral graphs by equilateral extended graphs. 
The approximation properties of the eigenvalues of the equilateral approximation as well as the convergence of a Newton-trace method with initial guesses obtained from the latter will then be investigated in the numerical experiments in Section \ref{sec:numerical_results}.
We conclude the article with a summary and ideas for future work.



\section{Background and Notation}\label{sec:background}
%
\subsection{Combinatorial Graphs}
A combinatorial graph $\G:=(\V,\E)$ is defined by a set of vertices $\V:= \lbrace v_1, \ldots, v_n \rbrace$ and a set of edges $\E:= \lbrace e_1, \ldots, e_m \rbrace$ where each edge is a pair of unordered, distinct vertices. 
If two distinct vertices $v_i,v_j$ are connected by an edge $e=(v_i,v_j),$ they are called adjacent and we write $v_i \sim v_j$.
The degree of a vertex $v$ is defined as the number of incident edges and denoted by $\deg(v)$. 
The set of edges incident to $v$ is denoted by $\E_v$.
In the scope of this article, we will always consider undirected, simple and connected graphs.
\\

Under these conditions, a combinatorial graph $\G$ can be described by its adjacency matrix $\Ad := (a_{ij})_{i,j=1, \ldots, n}$ with $a_{ij}=1$ if $v_i \sim v_j$ and $0$ otherwise.
Together with the degree matrix $\Deg := \diag(\deg(v_1), \ldots, \deg(v_n))$ it defines the {graph Laplacian matrix} $\Lapl:= \Deg - \Ad$.
We will later be interested in the \emph{harmonic graph Laplacian matrix} $\harmLaplG : = \Deg^{-1} \Lapl$. The harmonic graph Laplacian matrix $\harmLaplG$ is similar to the well-studied normalized graph Laplacian matrix $\normLapl:= \Deg^{-\frac{1}{2}} \Lapl  \Deg^{-\frac{1}{2}}$. This implies the following theorem on the eigenvalues of $\harmLaplG$, compare for example \cite{Chung_1997}.
\begin{thm}\label{th:eigvals_normalized_laplacian}
The harmonic graph Laplacian matrix $\harmLaplG$ and the normalized graph Laplacian matrix $\normLapl$ are similar with common eigenvalues 
$$
    0 = \eigvaldiscrete_1 < \eigvaldiscrete_2 \le \ldots \le \eigvaldiscrete_{n} \le 2. \quad 
$$
$\eigvaldiscrete_n = 2$ is a simple eigenvalue if and only if $\G$ is bipartite. 
\end{thm}
%
\subsection{Metric and Quantum Graphs}
A metric graph $\Gamma$ is defined by the combinatorial graph $\G=(\V,\E)$ by assigning a length $\elle \in \mathbb{R}^+$ to each edge $e \in \E$. 
The metric graph $\Gamma$ can be interpreted as a metric space where the edges are parameterized by their length $\ell_e$ and coupled at their common vertices. 
In this article, we always consider \emph{compact metric graphs}, meaning metric graphs that have finitely many edges with finite lengths $\elle < \infty$.
\\
A function $u$ on $\Gamma$ is defined as the collection of functions $u_e: [0,\ell_e] \to \R$ on the edges of $\Gamma$, possibly required to fulfill some kind of coupling conditions at the vertices. 
The expression $u(v)$ for $v \in \V$ and thereby the restriction $u_{\V}$ of $u$ to the vertices  is well defined whenever the function $u$ is \textit{continuous} on $\Gamma$.
Continuity in this context means that $u_e$ is continuous on the edges $e \in \E$ and $u_e(v)$ assumes the same value for all edges $e \in \E_{v}$.
For a more intensive discussion of the concept of metric graphs, we refer to \cite{Berkolaiko_Kuchment_2013}. 
\\

Due to its prominence in literature, we introduce the terminology \emph{quantum graph} which refers to a metric graph $\Gamma$ that is equipped with a differential operator $\Ham$ acting on a domain $\domHam$ of functions that fulfill certain coupling conditions at the vertices $v \in \V$. 
In the present article, the differential operator of interest is the negative second order derivative
\begin{equation}\label{eq:Ham_neg_sec_order_derivative}
\Hamu,	
\end{equation}
which should be understood as the negative second order derivative acting on each edge.
Moreover, we consider only the most prominent vertex coupling conditions, known as \emph{Neumann-Kirchhoff conditions}. They are given by the pair
	\begin{subequations}\label{eq:NK}
	\begin{align}
		& u \text{ is continuous on } \Gamma \label{eq:cont_condition} \\
		& \sumedgesincident \ddx{u_e} (v) = 0 \quad \text{for all } v \in \V \label{eq:Neum_Kirch_cond}
	\end{align}
	\end{subequations}
where the derivatives in \eqref{eq:Neum_Kirch_cond} are taken in the direction away from the vertex. 
We thus define
\begin{equation}
    \domHamNK := \bigoplus_{e \in \E} H^2(e) \, \cap \, \lbrace u \text{ fulfills Neumann-Kirchhoff conditions} \rbrace. 
    \vspace{-0.2cm}
\end{equation}
Here, $H^2(e) := H^2([0,\elle])$ denotes the Sobolev space of twice weakly differentiable functions on each edge. 
Neumann-Kirchhoff conditions are, for example, well-suited modeling diffusion problems on metric graphs, since \eqref{eq:cont_condition} models a continuity condition and \eqref{eq:Neum_Kirch_cond} a current conservation condition. 
Under these conditions, $\Ham$ is self-adjoint, see for example \cite{Berkolaiko_Kuchment_2013}, Theorem 1.4.4. 
\\

In the scope of this work, we only consider the differential operator $\Ham$, defined in \eqref{eq:Ham_neg_sec_order_derivative}, together with Neumann-Kirchhoff vertex conditions. Therefore, we always refer to a quantum graph as the triple
$$
\vspace*{-0.2cm}
    \lbrace \text{metric graph } \Gamma, \Hamu, \text{ Neumann-Kirchhoff conditions} \rbrace.
$$

\pagebreak
\section{Quantum Graph Spectra and connection to \linebreak Nonlinear Eigenvalue Problems}\label{sec:quantum_graph_spectra}
%
We are interested in the spectrum of the differential operator $\Ham$ acting on functions in $\domHamNK$. Since the operator $\Ham$ and the vertex coupling conditions in $\domHamNK$ are fixed, the spectrum of $\Ham$ only depends on the metric graph $\Gamma$. Therefore, we also speak of the spectrum of the quantum graph, denoted by $\specquantum$.
If $\Gamma$ is a compact metric graph, the spectrum of $\Ham$ is discrete, that is, it consists exclusively of isolated eigenvalues of finite multiplicity (\cite{Berkolaiko_Kuchment_2013}, Theorem 3.1.1.). 
Moreover, $\Ham$ is positive semidefinite with one simple zero eigenvalue corresponding to a constant eigenfunction.
This can be easily seen when considering the quadratic form of $\Ham$, given by (\cite{Berkolaiko_Kuchment_2013}, Theorem 1.4.11.)
$$
    \mathfrak{h}(u,u) = \sumedges \int_e \left(\ddx{u}\right)^2 \dup x \ge 0.
$$

Altogether, we deduce that the spectrum of a quantum graph consists exclusively of eigenvalues $\eigvalquantum \in \R, \eigvalquantum \geq 0$. Our objective is thus to solve the eigenvalue problem \vspace{-0.2cm}
\begin{equation}\label{eq:eigval_problem_quantum}
	   \Ham \eigfunc = \eigvalquantum \eigfunc
\end{equation}
with a nontrivial eigenfunction $\eigfunc \in \domHamNK$. 
\\

The following theorem from \cite{Kuchment_2003} (Theorem 19), which we slightly modified by giving an explicit characterization of $\bfH$, will be crucial for our further proceedings. 
\begin{thm}\label{th:NEP_vertex}
    Consider $\eigvalquantum \neq \NVEe \text{ for } k \in \N_0$ and $e \in \E$. Then $\eigvalquantum$ is an eigenvalue of $\Gamma$ if and only if there exists a nontrivial $\eigvec \in \R^n$ such that 
\begin{equation}\label{eq:NEP_H}
    \bfH(\eigvalquantum) \eigvec =0
\end{equation}
where the matrix $\bfH(\eigvalquantum) \in \R^{n \times n}$ is given by
\begin{equation}\label{eq:H_matrix}
	\bfH_{ij}(\eigvalquantum) :=
	\begin{cases}
		- \sum_{e \in \E_{v_i}} \cot \left( \eigvalquantumsqrt\, \elle \right) & \text{if } i=j\\
		 \frac{1}{\sin \left(\eigvalquantumsqrt \, \elle \right)} & \text{if } e =(v_i,v_j) \in \E
 		\\
		0 & \text{otherwise}
	\end{cases}. 
\end{equation}	
\end{thm}
\begin{proof}
We briefly sketch the proof according to the derivation of Theorem 19 in \cite{Kuchment_2003}.
For $\eigvalquantum > 0$, the solution of the eigenvalue problem \eqref{eq:eigval_problem_quantum} on each edge $e = (v_i,v_j)$ is given by
$$ 
	\eigfunc_e(x) = A_e \cos \left(\eigvalquantumsqrt \, x \right) + B_e \sin\left(\eigvalquantumsqrt \, x\right), \quad x \in [0, \elle] 
$$
 where $A_e, B_e \in \R$ are some edge specific constants. 
 If $\sin( \eigvalquantumsqrt \, \ell_e ) \neq 0$, the continuity condition implies that they can be expressed as
 \begin{align}\label{al:vertex_eigfunc_boundary}
		\eigfunce(x) = \frac{1}{\sin( \eigvalquantumsqrt \, \elle)} \left(\eigvec(v_i) \sin\left(\eigvalquantumsqrt( \elle-x)\right)+\eigvec(v_j)\sin\left(\eigvalquantumsqrt x\right) \right).
	\end{align}
By the current conservation condition \eqref{eq:Neum_Kirch_cond}, we deduce that
 a necessary and sufficient condition for $\eigfunc$ to be a nontrivial solution of the eigenvalue problem is
	\begin{equation}\label{eq:eigfunc_condition}
		\sum_{e \in \E_{v_i}} \eigfunc'_e(v_i) = 
			\sum_{ v_j \sim v_i} \frac{\eigvalquantumsqrt}{\sin(\eigvalquantumsqrt \, \elle)} 
			\left( (\eigvec(v_j) -  \cos\left(\eigvalquantumsqrt\,\elle\right) \eigvec(v_i) \right)  = 0
 	\end{equation}
 	for all $v_i \in \V$ where $\eigvec:= \eigfunc_{\V} \in \R^n$ is the restriction of $\eigfunc$ to vertices. 
 	Summarized in matrix-vector form, we obtain the assertion.
\end{proof}
Hence, Theorem \ref{th:NEP_vertex} implies that the eigenvalues  $\eigvalquantum \neq \NVEe \text{ for } k \in \N_0$ and $e \in \E$ of $\Gamma$ can be found as the solutions of a \emph{nonlinear eigenvalue problem} (NEP). 
In particular, the coefficient matrix $\bfH(\eigvalquantum)$ is only determined at the vertices of the underlying combinatorial graph, i.e., the NEP is of size $n \times n$. 
\begin{rmk}\label{re:nonvertex}
    In the derivation of \eqref{eq:eigfunc_condition}, we had to exclude possible eigenvalues with $\sin( \eigvalquantumsqrt \, \elle ) = 0$ and, consequently, Theorem \ref{th:NEP_vertex} only applies for $\eigvalquantum \neq \NVEe$ for all $k \in \N$ and $e \in \E$. 
    Since the eigenvalue problem in this situation reduces to a problem only defined on the vertices of $\Gamma,$ we will refer to these eigenvalues as \emph{vertex eigenvalues}.
    The remaining eigenvalues of the form $\eigvalquantum = \NVEe$ for $k \in \N$ and $e \in \E$ are referred to as \emph{non-vertex eigenvalues.} 
    To streamline the exposition, we will in this article only cover the vertex eigenvalues. 
    The non-vertex eigenvalues can be treated in a similar way after using a trick of inserting artificial vertices on the edges of $\Gamma,$ which is discussed intensively in \cite{Weller_2023}.
\end{rmk}
The NEP (\ref{eq:NEP_H}) has a nontrivial solution $\eigvec$ whenever $\bfH(\NEPvar)$ is singular, i.e., if $\NEPvar > 0$ is a root of $\det(\bfH(\NEPvar))$. 
To find the roots of $\det(\bfH(\NEPvar))$,
we will concentrate here on Newton methods such as the Newton-trace iteration (\cite{Lancaster_1966})
\begin{equation}\label{eq:New_trace_iteration}
    \NEPvar^{\newtonvar+1} = \NEPvar^{\newtonvar} - \frac{1}{\trace(\bfH^{-1}(\NEPvar^{\newtonvar}) \, \bfH'(\NEPvar^{\newtonvar}))} \,.
\end{equation}
We refer to \cite{Guettel_2017} for a review of other Newton methods applied for NEPs. 
Note that we introduced $\NEPvar$ to differentiate between the variable $\NEPvar$ and the eigenvalue $\eigvalquantum$.
\\
\\
In the rest of this work, the focus is not on the choice of the Newton iteration, but rather on the determination of suitable initial guesses, since these are the key to an efficient application of  \eqref{eq:New_trace_iteration} or a similar iteration.  
Our objective is to apply the eigenvalues of \emph{equilateral approximations} to start the Newton iteration. 
%
%
\pagebreak
\section{Approximation with Equilateral Graphs}\label{sec:equil_approximations}
%
A metric graph $\Gamma$ is called \emph{equilateral}, if all edges have the same length $\elle =: \ell$. 
In this special case, condition \eqref{eq:eigfunc_condition} for $\eigfunc$ to be a nontrivial solution of the eigenvalue problem \eqref{eq:eigval_problem_quantum} can be simplified to 
$$
    \sum_{v_j \sim v_i} \eigvec(v_j) - \cos\left(\eigvalquantumsqrt \, \ell\right)\deg(v_i)\eigvec(v_i) = 0 \quad \forall v_i \in \V.
$$
This leads to the well-known relation to the spectrum of the harmonic Laplacian matrix, see for example \cite{Berkolaiko_Kuchment_2013}, Theorem 3.6.1.
\begin{thm}\label{th:equilateral_spectrum}
	 $\eigvalquantum \neq \NVE$ for $k \in \N$ is an eigenvalue of $\Gamma$ if and only if 
	 $
	    \eigvaldiscrete = 1 - \cos(\eigvalquantumsqrt \, \ell )
	 $
	 is eigenvalue of the harmonic graph Laplacian matrix $\harmLaplG$.
\end{thm} 
Theorem \ref{th:equilateral_spectrum} implies that any vertex eigenvalue $\eigvalquantum$ can be determined by the rule
\begin{equation*}\label{eq:vertex_eigenvalue_rule}
	\eigvalquantum_{\eigvaldiscrete,k}=
	\begin{cases}
		\left( \frac{1}{\ell} \left( \arccos(1-\eigvaldiscrete)+k \pi \right) \right)^2 & \text{for } k \text{ even} 
		\\
		\left( \frac{1}{\ell} \left( \arccos(1-\eigvaldiscrete)-(k+1) \pi \right) \right)^2 & \text{for } k \text{ odd} 
	\end{cases},
\end{equation*}
for $k =1,2,3,\ldots$, i.e., by the solutions of a linear eigenvalue problem. 
Since $\normLapl$ (in contrast to $\harmLaplG$) is symmetric, we will always solve the symmetric eigenvalue problem $\normLapl \, \eigvec = \eigvaldiscrete \,\eigvec$  in the following numerical investigations. Due to the similarity of $\normLapl$ and $\harmLaplG$, the solutions can be transferred to $\harmLaplG$ .
\\

Although the restriction of Theorem \ref{th:equilateral_spectrum} to equilateral graphs is limiting, we will soon see how it is useful for approximating the eigenvalues of non-equilateral graphs. 
We will therefore first define a transformation of a graph which we refer to as \emph{cleaning} (this definition is adapted from \cite{Kurasov_2002}).
\begin{definition}
Let $\Gamma$ be a metric graph and $v \in \V$ be a vertex of degree two with two incident edges, say $e_1=(v,v_1)$ with length $\ell_1$ and $e_2=(v,v_2)$ with length $\ell_2$. Then, the cleaned edge is defined by eliminating the vertex of degree two and combining the two incident edges into one edge $(v_1,v_2)$ of length $\ell_1+\ell_2$.
\end{definition}
Accordingly, we define the cleaned graph $\cleangraph{\Gamma}$ where all vertices of degree two are eliminated by combining their two incident edges. 
Then, Remark 1.4.2. from \cite{Berkolaiko_Kuchment_2013} implies the following important theorem.
\begin{thm}\label{th:evals_cleaned}
The eigenvalues of a metric graph $\Gamma$ agree with the eigenvalues of its cleaned graph $\cleangraph{\Gamma}$.
\end{thm} 
\begin{proof}
We need to show that the eigenvalue problem posed on $\Gamma$ under Neumann-Kirchhoff conditions has the same solutions as the eigenvalue problem posed on $\cleangraph{\Gamma}$.
Let therefore $\eigfunc$ be an eigenfunction of $\Gamma$ and $v$ be a vertex of degree two in $\Gamma$. 
Then, the Neumann-Kirchhoff conditions guarantee the continuity of $\eigfunc$ and its first derivative at $v$, i.e., the two $H^2$-pieces of $\eigfunc$ match into one $H^2$-function on the cleaned edge. 
Thus, each eigenfunction of $\Gamma$ is also a solution of the eigenvalue problem on $\cleangraph{\Gamma}$ and vice versa.
\end{proof}
On the other hand, we can also perform an opposite transformation by inserting artificial vertices of degree two on the edges. We will then speak of \emph{extended graphs}\footnote{\tiny The term extended graph is used in reference to \cite{Arioli_2018}, where discretization nodes along edges are considered as additional vertices.}.
\\

If now, for example, $\ell_e \in \mathbb{N}$, it is possible to represent $\Gamma$ by the extended graph $\Gammaext_{\text{gcd}}$ where we add artificial vertices so that $\Gammaext_{\text{gcd}}$ has equilateral edge length according to the greatest common divisor of the edge lengths $\elle,$ as shown in Figure \ref{fig:extended_equil_graph_representation}. 
Theorem \ref{th:evals_cleaned} implies that $\Gamma$ and $\Gammaext_{\text{gcd}}$ have the same eigenvalues, allowing us to apply Theorem \ref{th:equilateral_spectrum} to the equilateral graph $\tilde{\Gamma}_{\text{gcd}}$ in order to obtain the eigenvalues of $\Gamma$. 
This strategy can be extended to graphs with edge lengths $\ell_e \in \mathbb{Q}$. 
However, if $\elle \not \in \mathbb{Q}$ for at least one edge, such a representation is no longer possible. 
We will instead try to find an approximation of $\Gamma$ by an equilateral graph $\Gammaapprox$. This can be obtained by approximating all edge lengths with values in $\mathbb{Q}$ and then proceeding as described above. We illustrate this procedure in Figure \ref{fig:equilateral_approximation_graph}.
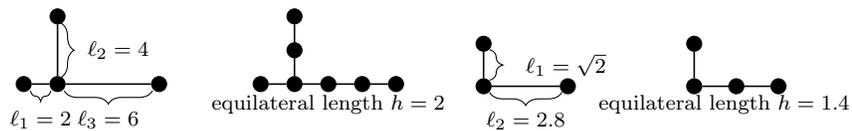
\begin{figure}[h]
	\center
	\begin{subfigure}[]{0.45\textwidth}
		\center
		\usetikzlibrary{decorations.pathreplacing}
		\begin{tikzpicture}[scale=0.45]
		\node[shape=circle,inner sep=0pt, text width=2mm,draw=black,fill=black] (A) at (0,0)  {}; 
		\node[shape=circle,draw=black,inner sep=0pt, text width=2mm,fill=black] (B) at (0,2)  {}; 
		\node[shape=circle,draw=black,inner sep=0pt, text width=2mm,fill=black] (C) at (3,0)  {}; 
		\node[shape=circle,draw=black,inner sep=0pt, text width=2mm,fill=black] (D) at (-1,0) {}; 
		\path [semithick,-] (A) edge node[above=5pt] {} (D);
		\path [semithick,-] (A) edge node[above=5pt] {} (B);
		\path [semithick,-] (A) edge node[above=5pt] {} (C);
		\draw [decorate,decoration={brace,mirror,amplitude=4pt}] (-0.8,-0.2) -- (-0.2,-0.2) node [midway,yshift=-0.15in] {\footnotesize $\ell_1=2$};
		\draw [decorate,decoration={brace,mirror,amplitude=5pt}] (0.2,-0.2) -- (2.8,-0.2) node [midway,yshift=-0.15in] {\footnotesize $\ell_3=6$};
		\draw [decorate,decoration={brace,mirror,amplitude=5pt}] (0.1,0.2) -- (0.1,1.8) node [midway,xshift=0.3in] {\footnotesize $\ell_2=4$};
		\node[shape=circle,inner sep=0pt, text width=2mm,draw=black,fill=black] (A1) at (7,0)  {}; 
		\node[shape=circle,draw=black,inner sep=0pt, text width=2mm,fill=black] (B1) at (7,2)  {}; 
		\node[shape=circle,draw=black,inner sep=0pt, text width=2mm,fill=black] (C1) at (10,0)  {}; 
		\node[shape=circle,draw=black,inner sep=0pt, text width=2mm,fill=black] (D1) at (6,0) {}; 
		\node[shape=circle,inner sep=0pt, text width=2mm,draw=black,fill=black] (AD) at (7,1)  {}; 
		\node[shape=circle,inner sep=0pt, text width=2mm,draw=black,fill=black] (AD) at (8,0)  {}; 
		\node[shape=circle,inner sep=0pt, text width=2mm,draw=black,fill=black] (AD) at (9,0)  {}; 
		\path [semithick,-] (A1) edge node[above=5pt] {} (D1);
		\path [semithick,-] (A1) edge node[above=5pt] {} (B1);
		\path [semithick,-] (A1) edge node[above=5pt] {} (C1);
		\draw (8,-0.6) node {\footnotesize equilateral length $\equilength=2$};
		\end{tikzpicture}
		\caption{Equilateral representation of a graph with $\elle \in \N$.}\label{fig:extended_equil_graph_representation}
	\end{subfigure}
	\begin{subfigure}[]{0.54\textwidth}
		\center
		\usetikzlibrary{decorations.pathreplacing}
		\begin{tikzpicture}[scale=0.4]
		\node[shape=circle,inner sep=0pt, text width=2mm,draw=black,fill=black] (A) at (0,0+0.7)  {}; 
		\node[shape=circle,draw=black,inner sep=0pt, text width=2mm,fill=black] (B) at (0.0,1.4142135623730951+0.7)  {}; 
		\node[shape=circle,draw=black,inner sep=0pt, text width=2mm,fill=black] (C) at (2.8,0+0.7)  {}; 
		\path [semithick,-] (A) edge node[above=5pt] {} (B);
		\path [semithick,-] (A) edge node[above=5pt] {} (C);
		\draw [decorate,decoration={brace,mirror,amplitude=5pt}] (0.2,-0.2+0.7) -- (2.6,-0.2+0.7) node [midway,yshift=-0.15in] {\footnotesize $\ell_2=2.8$};
		\draw [decorate,decoration={brace,amplitude=5pt,mirror}] (0.2,0.2+0.7) -- (0.2,1.4142135623730951-0.2+0.7) node [midway,xshift=+0.4in] {\footnotesize $\ell_1=\sqrt{2}$};
		\node[shape=circle,inner sep=0pt, text width=2mm,draw=black,fill=black] (A) at (7,+0.7)  {}; 
		\node[shape=circle,draw=black,inner sep=0pt, text width=2mm,fill=black] (B) at (7,1.4142135623730951+0.7)  {}; 
		\node[shape=circle,draw=black,inner sep=0pt, text width=2mm,fill=black] (C) at (9.8,0+0.7)  {}; 
		\node[shape=circle,draw=black,inner sep=0pt, text width=2mm,fill=black] (E) at (8.4,0+0.7)  {}; 
		\path [semithick,-] (A) edge node[above=5pt] {} (B);
		\path [semithick,-] (A) edge node[above=5pt] {} (C);
		\draw (8,-0.6+0.7) node {\footnotesize equilateral length $\equilength=1.4$};
		\end{tikzpicture}
	\caption{Equilateral approximation graph.}\label{fig:equilateral_approximation_graph}
	\end{subfigure}
	\caption{Examples of an equilateral representation and approximation.}
\end{figure}
\\
To formalize the idea, we define an equilateral approximation of $\Gamma$ as follows.
\begin{definition}
Let $\Gamma$ be a non-equilateral metric graph with underlying combinatorial graph $\G$ and edge lengths $\ellvec \in \mathbb{R}^m$. 
An \emph{equilateral approximation} of $\Gamma$ is a graph $\Gammaapprox_{\equilength}$ with equilateral edge length $\equilength$ such that $\cleangraph{\Gammaapprox_{\equilength}}$ has underlying combinatorial graph $\G$ and edge lengths $\ellapprox_{\equilength} \approx \ellvec$.
\end{definition}
Accordingly, we define the \emph{distance} between $\Gamma$ and $\Gammaapprox_{\equilength}$.
\begin{definition}
Let $\Gamma$ be a non-equilateral metric graph with edge lengths $\ellvec$ and  $\Gammaapprox_{\equilength}$ be an equilateral approximation with cleaned edge lengths $\ellapprox_{\equilength}$. Then, the \emph{distance} between $\Gamma$ and $\Gammaapprox_{\equilength}$ is defined as
$$
    \dist(\Gamma,\Gammaapprox_{\equilength}) := \|\ellvec - \ellapprox_{\equilength} \|
$$
with $\| \cdot \| = \| \cdot \|_2$. 
\end{definition}
Note that $\Gamma$ and $\cleangraph{\Gammaapprox_{\equilength}}$ have the same edges, i.e., $\ellvec - \ellapprox_{\equilength}$ is well defined. 
In the definition of equilateral approximations, we have intentionally chosen a somewhat vague formulation $\ellapprox_{\equilength} \approx \ellvec$, since it will be a main question arising from this work, how close an approximation must be to deliver appropriate eigenvalue estimates.
Clearly, we can always find $\Gammaapprox_{\equilength}$ with $\dist(\Gamma,\Gammaapprox_{\equilength}) \to 0$ for $\equilength \to 0.$ 
\\

The following types of equilateral approximations are of special interest.
\begin{definition}
An equilateral approximation of $\Gamma$ is referred to as \emph{equilateral floor approximation} $\Gammaapproxfloorh{\equilength}$ if $\ellapprox_{\equilength}-\ellvec \le \bfzero$, i.e., if any edge length of the approximation is shorter than the edge length of the exact graph. 
Equivalently, we use the terminology \emph{equilateral ceil approximation} $\Gammaapproxceilh{\equilength}$ if $\ellapprox_{\equilength}-\ellvec \ge \bfzero$.
\end{definition}
The numerical experiments will reveal that the eigenvalues of the equilateral floor and ceil approximations are upper and lower bounds for the eigenvalues of $\Gamma$. 
In particular, it will become apparent that for decreasing $\equilength$, both bounds approach the eigenvalues of $\Gamma$. 
This motivates the idea to apply the eigenvalues of equilateral floor and ceil approximations as starting values for a Newton iteration.
A key question will be how $\dist(\Gamma,\Gammaapprox_{\equilength})$ and thus $\equilength$ influences the convergence of the Newton method, i.e., how precise the initial guesses must be to guarantee a fast convergence. 
\section{Numerical Results}\label{sec:numerical_results}
%

The presented numerical experiments have been implemented in \texttt{Julia}, version 1.8.0 using example graphs from the \texttt{Graphs} package and eigenvalue algorithms from \texttt{LinearAlgebra} and \texttt{Arpack}. The plots have been generated with \texttt{Plots}.
\subsection{Equilateral Approximations}
For a given non-equilateral metric graph $\Gamma$, we compute equilateral floor and ceil approximations and investigate the approximation properties of their eigenvalues.
The computation of a sequence of equilateral floor approximations with $\dist(\Gamma,\Gammaapprox_{\equilength}) \to 0$ is summarized in Algorithm \ref{alg:floor_equi_approximation} and illustrated in Figure \ref{fig:equi_floor_approximation}.
A sequence of equilateral ceil approximations can be computed equivalently by setting $\noequisubdiv_e = \texttt{ceil}(\elle / \equilength)$ in line \ref{alg:line:floor}. 
 \begin{algorithm}[H]
    \caption{Sequence of Equilateral Floor Approximations}\label{alg:floor_equi_approximation}
    \begin{algorithmic}[1]
		\For{$J=1,2, \ldots$}
		\State Set $h = 2^{-J}$
		\State Initialize $\ellapprox \in \R^m$
		\For{$e \in \E$}
			\State Compute $\noequisubdiv_e = \texttt{floor}(\elle / \equilength)$ \label{alg:line:floor}
			\State Set $(\ellapprox)_e = \equilength \cdot \noequisubdiv_e$
		\EndFor
		\State Compute $\Gammaapprox_{\text{fl},\equilength}$ as extended graph of $\Gamma$ with edge length $\ellapprox_{\equilength}$ 
		\EndFor
	\end{algorithmic}
  \end{algorithm}
\begin{figure}[h]
	\center
	\begin{subfigure}[]{1\textwidth}
		\center
		\usetikzlibrary{decorations.pathreplacing}
		\begin{tikzpicture}[scale=0.8]
		\node[shape=circle,inner sep=0pt, text width=2mm,draw=black,fill=black] (A) at (0,0)  {}; 
		\node[shape=circle,draw=black,inner sep=0pt, text width=2mm,fill=black] (B) at (0,2.4)  {}; 
		\node[shape=circle,draw=black,inner sep=0pt, text width=2mm,fill=black] (C) at (1.8,0)  {}; 
		\path [semithick,-] (A) edge node[above=5pt] {} (B);
		\path [semithick,-] (A) edge node[above=5pt] {} (C);
		\draw [decorate,decoration={brace,mirror,amplitude=5pt}] (0.2,-0.2) -- (1.6,-0.2) node [midway,yshift=-0.15in] {\scriptsize $\ell_1=0.9$};
		\draw [decorate,decoration={brace,mirror,amplitude=5pt}] (0.1,0.2) -- (0.1,2.2) node [midway,xshift=0.3in] {\scriptsize $\,\, \ell_2=1.2$};
		\draw (1,-1.5) node {\scriptsize non-equilateral graph $\Gamma$};
		\node[shape=circle,inner sep=0pt, text width=2mm,draw=black,fill=black] (A1) at (0+4,0)  {}; 
		\node[shape=circle,inner sep=0pt, text width=1mm,draw=black,fill=black] (A1B1) at (0+5,0)  {}; 
		\node[shape=circle,draw=black,inner sep=0pt, text width=1mm,fill=black] (B1) at (0+4,2)  {}; 
		\node[shape=circle,draw=black,inner sep=0pt, text width=1mm,fill=black] (B1C1) at (0+4,2-1)  {}; 
		\path [semithick,-] (A1) edge node[above=5pt] {} (B1);
		\path [semithick,-] (A1) edge node[above=5pt] {} (A1B1);
		\draw [decorate,decoration={brace,mirror,amplitude=5pt}] (0.2+3.8,-0.2) -- (1.8+3.3,-0.2) node [midway,yshift=-0.15in] {\scriptsize $\ellapprox_1=0.5$};
		\draw [decorate,decoration={brace,mirror,amplitude=5pt}] (0.1+4,0.2) -- (0.1+4,1.8) node [midway,xshift=0.3in] {\scriptsize $\,\ellapprox_2=1$};
		\draw (1+4-0.4,-1.5) node {\scriptsize $h=0.5$};
		\node[shape=circle,inner sep=0pt, text width=2mm,draw=black,fill=black] (A2) at (0+7.5,0)  {}; 
		\node[shape=circle,inner sep=0pt, text width=1mm,draw=black,fill=black] (A2B2) at (0+7.5+0.5,0)  {}; 
		\node[shape=circle,inner sep=0pt, text width=1mm,draw=black,fill=black] (A2B2) at (0+7.5+1,0)  {}; 
		\node[shape=circle,inner sep=0pt, text width=1mm,draw=black,fill=black] (A2B2) at (0+7.5+1.5,0)  {}; 
		\node[shape=circle,draw=black,inner sep=0pt, text width=1mm,fill=black] (B22) at (0+7.5,2)  {}; 
		\node[shape=circle,draw=black,inner sep=0pt, text width=1mm,fill=black] (B221) at (0+7.5,1.5)  {}; 
		\node[shape=circle,draw=black,inner sep=0pt, text width=1mm,fill=black] (B222) at (0+7.5,0.5)  {}; 
		\node[shape=circle,draw=black,inner sep=0pt, text width=1mm,fill=black] (B2C2) at (0+7.5,2-1)  {}; 
		\path [semithick,-] (A2) edge node[above=5pt] {} (B22);
		\path [semithick,-] (A2) edge node[above=5pt] {} (A2B2);
		
		\draw [decorate,decoration={brace,mirror,amplitude=5pt}] (0.2+7.5,-0.2) -- (1.8+7.1,-0.2) node [midway,yshift=-0.15in] {\scriptsize $\ellapprox_1=0.75$};
		\draw [decorate,decoration={brace,mirror,amplitude=5pt}] (0.1+7.5,0.2) -- (0.1+7.5,2) node [midway,xshift=0.3in] {\footnotesize $\quad \ellapprox_2=1$};
		\draw (1+7.5-0.2,-1.5) node {\scriptsize $h=0.25$};
		%
		%
		\node[shape=circle,inner sep=0pt, text width=2mm,draw=black,fill=black] (A3) at (0+11,0)  {}; 
		\node[shape=circle,draw=black,inner sep=0pt, text width=1mm,fill=black] (B3) at (0+11,2.22)  {}; 
		\node[shape=circle,draw=black,inner sep=0pt, text width=1mm,fill=black] (C3) at (1.75+11,0)  {}; 
				
		\path [thick,-,dashed] (A3) edge node[above=5pt] {} (B3);
		\path [thick,-,dashed] (A3) edge node[above=5pt] {} (C3);
		
		\draw [decorate,decoration={brace,mirror,amplitude=5pt}] (0.2+11,-0.2) -- (1.6+11,-0.2) node [midway,yshift=-0.15in] {\scriptsize $\ellapprox_1=0.875$};
		\draw [decorate,decoration={brace,mirror,amplitude=5pt}] (0.1+11,0.2) -- (0.1+11,2.22) node [midway,xshift=0.3in] {\footnotesize $\quad \ellapprox_2=1.125$};
		\draw (1+11,-1.5) node {\scriptsize $\equilength=0.125$};
		\draw (0,-1) node {};
		\end{tikzpicture}
	\end{subfigure}
	\caption{Sequence of equilateral floor approximations with $\dist(\Gamma,\Gammaapproxfloorh{\equilength}) \to 0$.}
	\label{fig:equi_floor_approximation}
\end{figure}
In the following experiment, we will consider four different example graphs, each equipped with randomly chosen edge lengths $\elle \in [1,2]$, rounded to three decimal digits. 
The example graphs are a star graph with $n=5$ vertices and $m=4$ edges, a diamond graph with $n=4$ vertices and $m=5$ edges, a cycle graph with $n=4$ vertices and $m=4$ edges, and a Barab\'asi-Albert graph (\cite{Barabasi_1999}) with $n=10$ vertices and $m=16$ edges. 
Since the edge lengths are rounded to three decimal digits, the exact eigenvalues can be computed using an equilateral representation with edge lengths $\equilength = 0.001$.
\\
We compare the first non-zero eigenvalue $\eigvalquantum_2(\Gamma)$ of $\Gamma$ to the first non-zero eigenvalues $\eigvalquantum_2(\Gammaapproxfloorh{\equilength})$, $\eigvalquantum_2(\Gammaapproxceilh{\equilength})$ of the equilateral floor and ceil approximations for $\equilength=2^{-J}$, $J=1,\ldots,8$.
In Figure \ref{fig:non_equi_spectral_gap_floor_ceil}, the resulting distances $\dist(\Gamma,\Gammaapproxfloorh{\equilength})$ and $\dist(\Gamma,\Gammaapproxceilh{\equilength})$ are plotted against the absolute difference of the corresponding eigenvalues. The results illustrate that the eigenvalues of the equilateral approximations converge to the exact eigenvalues for all analyzed examples. 
\begin{figure}[h]
	\centering
	\includegraphics[width=0.49\textwidth]{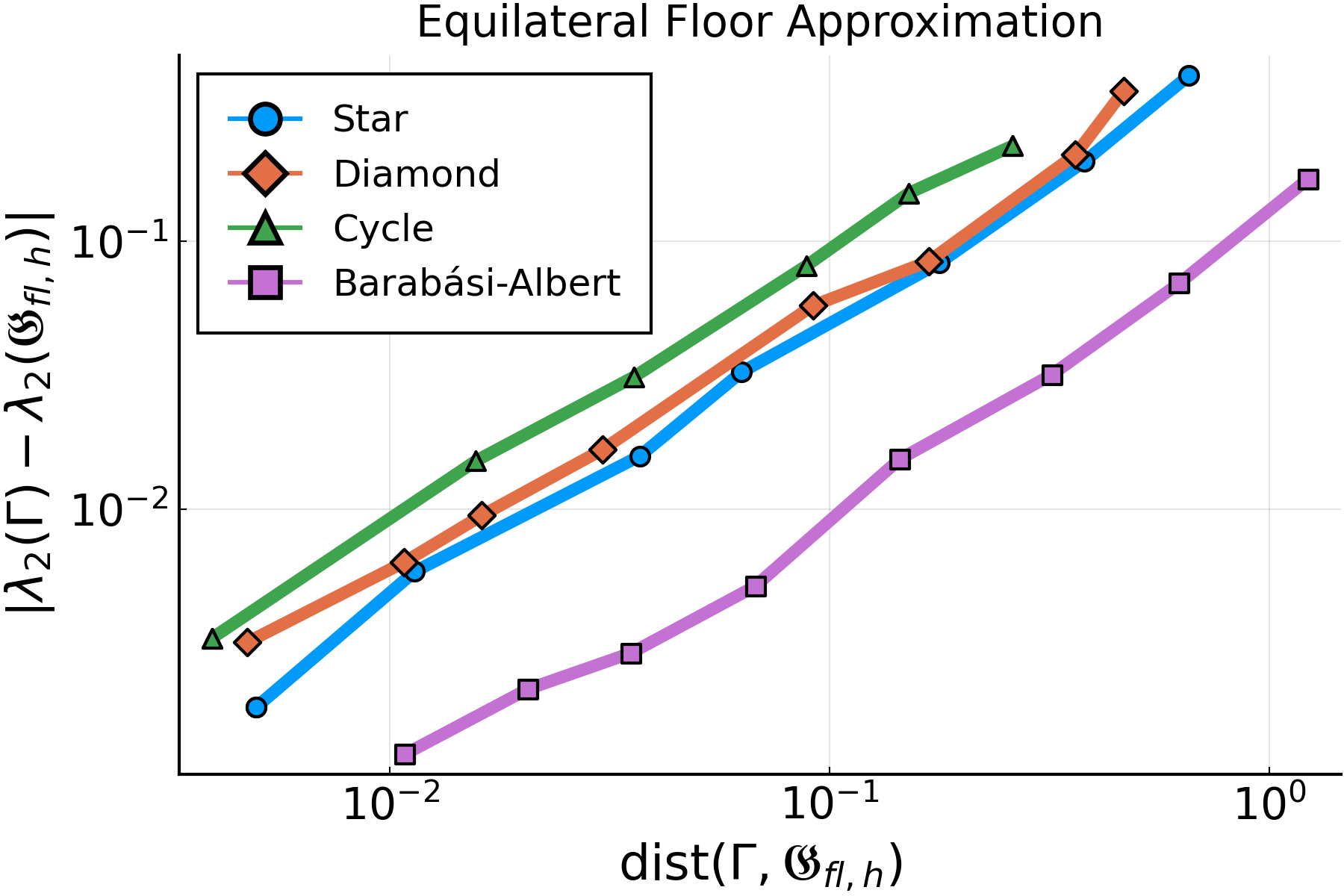}
	\includegraphics[width=0.49\textwidth]{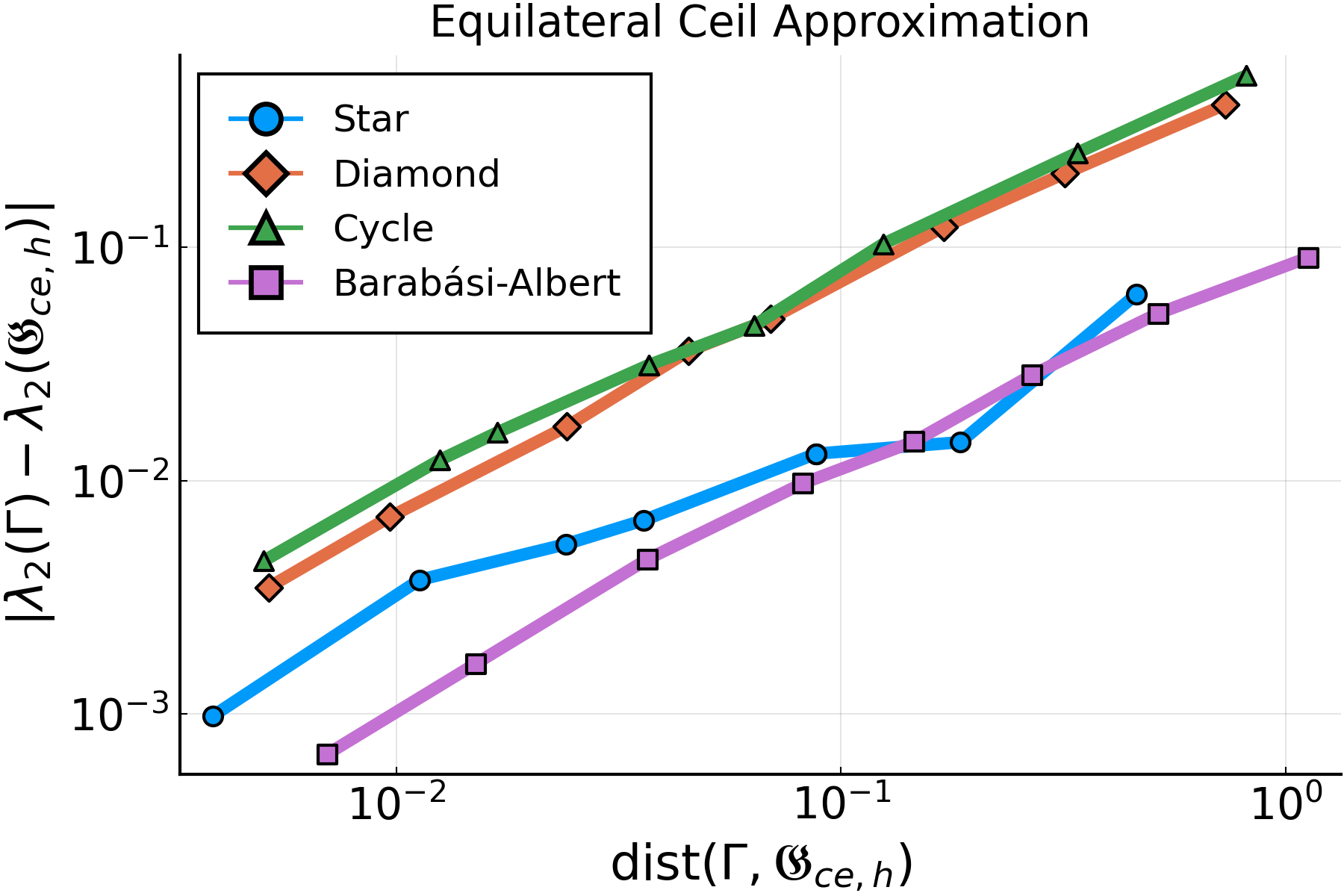}
	\caption{Approximating the spectral gap of the four example graphs via equilateral floor and ceil approximations.}
	\label{fig:non_equi_spectral_gap_floor_ceil}
\end{figure}
\\

Moreover, the first 50 eigenvalues of the exact graphs are compared with the first 50 eigenvalues of $\Gammaapprox_{\text{fl},\equilength}$ and $\Gammaapprox_{\text{ce},\equilength}$. The absolute deviation  $|\eigvalquantum_i(\Gamma)-\eigvalquantum_i(\Gammaapprox_{\equilength})|$ is exemplary illustrated for $\equilength= 2^{-4},2^{-6},2^{-8}$ and all four example graphs in Figure \ref{fig:non_equi_eigvals_floor_ceil_50}.
We observe that for any considered example, the ceil approximations approach the eigenvalues from below while the eigenvalues of the floor approximations are always higher than the exact ones.
Consistent with the results in Figure \ref{fig:non_equi_spectral_gap_floor_ceil}, the accuracy of the eigenvalue estimates increases for decreasing $\equilength$.
\begin{figure}[h]
	\centering
	\includegraphics[width=0.85\textwidth]{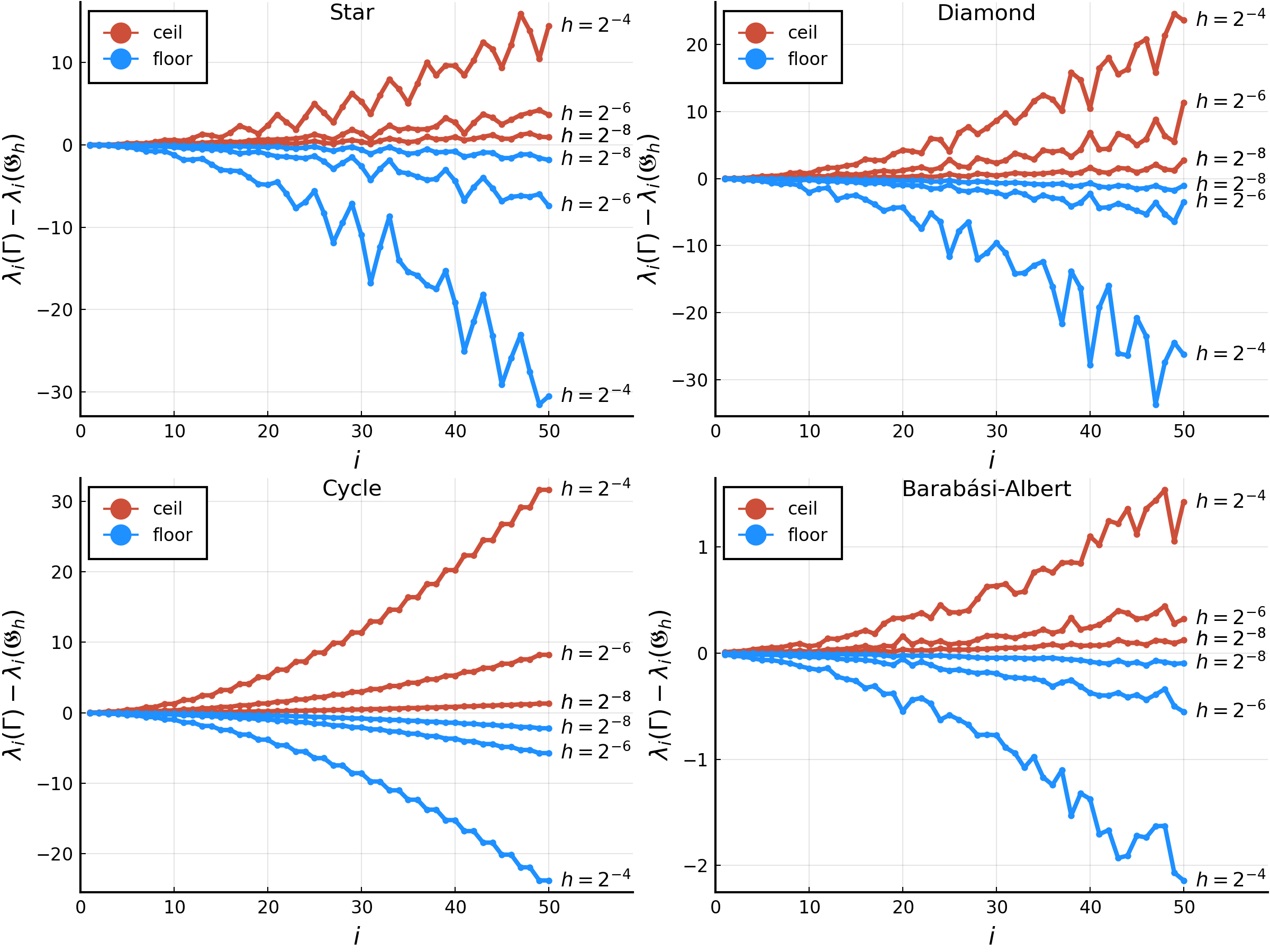}
	\caption{Equilateral floor and ceil approximation of the first 50 eigenvalues of the four example graphs.}
	\label{fig:non_equi_eigvals_floor_ceil_50}
\end{figure}
\subsection{Newton-Trace Iteration}
%
As shown in the previous experiment, the eigenvalues of the equilateral approximations converge to the exact eigenvalues for all analyzed examples. 
However, as mentioned earlier in Section \ref{sec:quantum_graph_spectra}, our aim is not to approximate the exact eigenvalues of a given non-equilateral metric graph $\Gamma$ by computing arbitrary close equilateral approximations, but only to find suitable estimates to start a subsequent Newton iteration. 
\\

The objective of the following experiment is to test the applicability of these eigenvalue estimates as starting values for a Newton iteration as well as the influence of the accuracy of the estimate on the number of required Newton iterations. 
As in the previous section, the experiments will be conducted to compute the first positive eigenvalue $\eigvalquantum_2$.
\\
We will consider two different sample graphs, each with random edge lengths $\elle \in [1,2]$, rounded to two decimal digits. 
These include a star graph with $n = 6$ vertices and $m = 5$ edges and a  Barab\'asi-Albert graph with $n = 50$ vertices and $m = 96$ edges. 
As the edge lengths are rounded to two decimal digits, the exact eigenvalues can be computed through an equilateral representation with edge length $\equilength = 0.01$.
\\
Since we have observed (Figure \ref{fig:non_equi_eigvals_floor_ceil_50}) that the exact eigenvalues lie between the eigenvalues of the floor and ceil approximation, we use the mean of the eigenvalues as the initial guess for the Newton-trace iteration, i.e., 
\begin{equation}\label{eq:lambda_start}
	\eigvalquantum^{\textup{init}} := \frac{\eigvalquantum_2(\Gammaapproxfloorh{\equilength}) + \eigvalquantum_2(\Gammaapproxceilh{\equilength})}{2}.		
\end{equation}
Here, as before, $\lambda_2(\Gammaapproxfloorh{\equilength})$ denotes the first positive eigenvalue of the equilateral floor approximation and $\lambda_2(\Gammaapproxceilh{\equilength})$ of the equilateral ceil approximation. 
In the following experiment, we now want to investigate how many Newton iterations are required when starting with $\lambda^{\textup{init}}$ as defined in $\eqref{eq:lambda_start}$ for various $h=2^{-J}, J=2, \ldots, 6$.\\

As a stopping criterion for the Newton-trace iteration, we conduct the reciprocal condition number of $\bfH(z)$ defined as
$$ 
\kappa^{-1}(\bfH) := \frac{1}{\|\bfH\| \ \|\bfH^{-1}\| }.
$$
Since $\bfH(z)$ is singular at the eigenvalues, a reciprocal condition number close to zero indicates that the Newton-trace iteration converged to an eigenvalue $\eigvalquantum$. 
In the following experiment, we stop the Newton-trace iteration whenever $\kappa^{-1}(\bfH) < 10^{-10}$ or the number of iterations $N_{\textup{iter}}$ exceeds 1000.

\begin{table}[h]
	\centering
	\begin{tabular}{||l || c c c ||c c c ||} 
	 \hline
	  & \multicolumn{3}{|c||}{Star, $\eigvalquantum_2 \approx 0.701372$} & \multicolumn{3}{|c||}{Barab\'asi-Albert, $\eigvalquantum_2 \approx 0.212386$ } \\[0.5ex]
		\hline
	 $\equilength$ & $\lambda^{\text{init}}$ & $N_{\text{iter}}$ & $\eigvalquantum_{\text{NEP}}$ & $\lambda^{\text{init}}$ & $N_{\text{iter}}$ & $\eigvalquantum_{\text{NEP}}$ \\ [0.5ex] 
	 \hline\hline
	 0.25 & 0.785305 & 8 & 0.701372 & 0.221519 & 11 & 0.212386\\ 
	 \hline
	 0.125 & 0.725887 & 4 & 0.701372 & 0.213850 & 4 & 0.212386 \\
	 \hline
	 0.0625 & 0.710570 & 3 & 0.701372  & 0.211420 & 4 & 0.212386  \\
	 \hline
	 0.03125 & 0.698711 & 3 & 0.701372 & 0.212036 & 3 & 0.212386\\
	 \hline
	 0.015625 & 0.704270 & 3 & 0.701372 & 0.212253 & 3 & 0.212386 \\ 
	 \hline 
	\end{tabular}
	\caption{Number of Newton-trace iterations ($N_{\text{iter}}$) required for computing $\eigvalquantum_2$ for the star graph with $n = 6$ vertices and $m = 5$ edges, as well as for the Barab\'asi-Albert graph with $n = 50$ vertices and $m = 96$ edges.}
	\label{tab:number_of_newton_iterations}
\end{table}

In Table \ref{tab:number_of_newton_iterations}, we have recorded the exact eigenvalue $\eigvalquantum_2$, the initial value $\eigvalquantum^{\textup{init}}$ arising from the equilateral approximations with length $\equilength$ as well as the number of required Newton-trace iterations until convergence to $\eigvalquantum_{\textup{NEP}}$.
First of all, it is important to notice that for each initial value $\eigvalquantum^{\textup{init}}$, the Newton-trace iteration converges to the same eigenvalue $\eigvalquantum_{\textup{NEP}}$. 
Given that the $\bfH(z)$-matrix in the Newton-trace iteration is only of size $n \times n$, the number of required iterations is moderate, even for the ``worst'' choice of $\eigvalquantum^{\textup{init}}$. 
As expected from the results in the previous subsections, an decreasing edge length $\equilength$ leads to better initial values, which in turn further reduces the number of required Newton-trace iterations.

%
\subsection{Nested Iteration Approach}
So far we have only studied the convergence towards the smallest non-zero eigenvalue. 
The question is whether further eigenvalues can still be found reliably by equilateral approximations, i.e., if patterns in the spectrum can be reflected sufficiently.
In other words, how can we guarantee that all eigenvalues are found?
\\
Moreover, for the previous experiments, we have consulted example graphs of moderate size and with a limited number of decimal digits for the edge lengths. This guaranteed, that the eigenvalue problem of the equilateral approximations can be easily solved by standard eigenvalue algorithms. 
Other phenomena might occur when considering more complex, large graphs.
In general, for small step sizes $\equilength$, this results in large scale linear eigenvalue problems
\begin{equation}\label{eq:ext_lapl_system}
  	\normLapl_{\equilength} \, \eigvec_{\equilength} = \eigvaldiscrete_{\equilength} \,\eigvec_{\equilength}  
\end{equation}
where $\normLapl_{\equilength}$ is the normalized graph Laplacian matrix of the combinatorial graph of $\Gammaapprox_{\equilength}$.
These problems can quickly exceed a size practicable for a classical eigenvalue solver. 
We therefore propose to follow a \emph{nested iteration approach} for the solution of \eqref{eq:ext_lapl_system}. 
The idea is that the eigenvalues $\eigvalquantum_{\equilength}$ of $\Gammaapprox_{\equilength}$ are good approximations for eigenvalues $\eigvalquantum_{\equilength/2}$ of $\Gammaapprox_{\equilength/2}$. 
Thus, we apply an inverse iteration with shift $1-\cos\left(\sqrt{\eigvalquantum_{\equilength}} \, \equilength/2 \right)$ as eigenvalue solver for
$$
\normLapl_{\equilength/2} \, \eigvec_{\equilength/2} = \eigvaldiscrete_{\equilength/2} \,\eigvec_{\equilength/2}
$$
and so on.
\\

With this nested iteration eigenvalue solver, we will finally consider a random Barab\'asi-Albert graph with $n=500$ vertices and $m=1491$ edges with a randomly assigned edge length $\elle \in [1,5]$. 
For $\equilength = 2^{-J}, J=1,\ldots,8$, we each time compute the $Q=10$ smallest eigenvalues of equilateral floor and ceil approximations $\Gammaapproxceilh{\equilength},\Gammaapproxfloorh{\equilength}$ and run the Newton-trace iteration with initial value $\eigvalquantum_q^{\text{init}} = \frac{1}{2} \left( \eigvalquantum_q(\Gammaapproxfloorh{\equilength}) + \eigvalquantum_q(\Gammaapproxceilh{\equilength}) \right).$
\\
Since $\elle \not \in \mathbb{Q}$, we cannot compute the exact solution.
Instead, we plot the reciprocal condition number of $\bfH(z)$, which roots indicate the eigenvalues as illustrated in the upper plot in Figure \ref{fig:rec_cond_evals}.
\begin{figure}[h]
    \centering
    \includegraphics[width=0.9\textwidth]{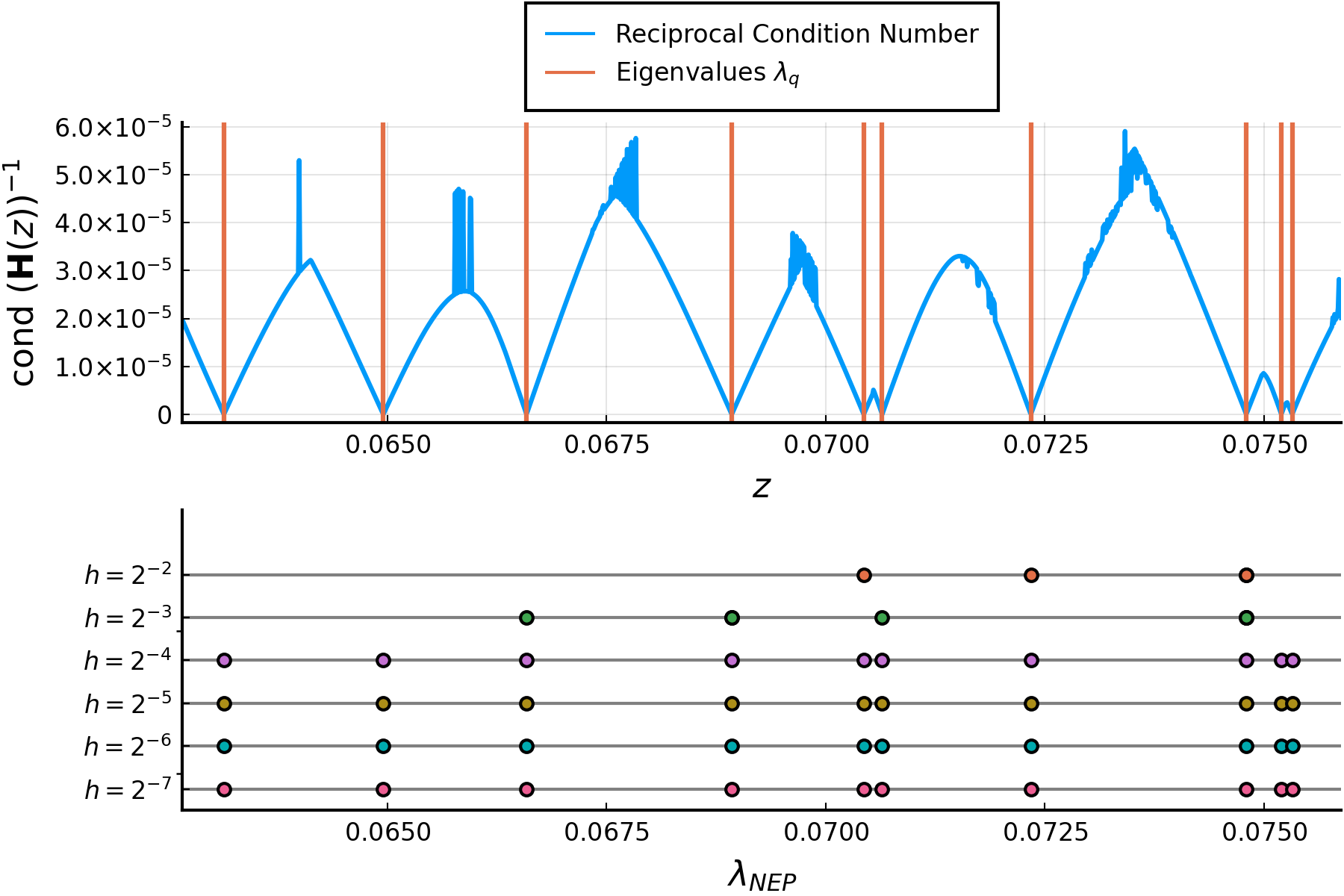}
    \caption{Reciprocal condition number of $\bfH(z)$ for a Barábasi-Albert graph with $n=500$ vertices compared to the computed eigenvalues by equilateral approximations with $\Gammaapproxceilh{\equilength},\Gammaapproxfloorh{\equilength}$ and a subsequent Newton-trace iteration.}
    \label{fig:rec_cond_evals}
\end{figure}
Observe that the spacing between the roots is very irregular. 
In the lower part of Figure \ref{fig:rec_cond_evals}, the eigenvalues found by the Newton iteration with $\eigvalquantum_i^{\text{init}}$ are plotted for the different levels $J=1,\ldots,8$. 
Clearly, a large step size is not sufficient to resolve all patterns in the spectrum, as not all eigenvalues can be found. 
In particular, we observed that this effect appears with large, sparse graphs. 
As a criterion to estimate the quality of the equilateral approximations, we suggest to verify that $\eigvalquantum_q(\Gammaapproxfloorh{\equilength}) > \eigvalquantum_q(\Gammaapproxceilh{\equilength})$ for $q=1,\ldots,Q$.
Finally, even for large graphs, the number of required Newton iterations remains moderate, as shown in Table \ref{tab:Newton_iter}.
\begin{table}[h]
\centering
\begin{tabular}{|c|c|c|c|c|c|c|c|c|c|c|}
\hline
& $\lambda_1$ & $\lambda_2$ & $\lambda_3$ & $\lambda_4$ & $\lambda_5$ & $\lambda_6$ & $\lambda_7$ & $\lambda_8$ & $\lambda_9$ & $\lambda_{10}$ \\ 
\hline
\hline
$\equilength = 2^{-2}$ & - & - & - & - & 11 & - & 9 & 18 & - & - \\ \hline
$\equilength = 2^{-3}$ & - & - & 10 & 16 & - & 3 & - & 24 & - & - \\ \hline
$\equilength = 2^{-4}$ & 4 & 2 & 3 & 3 & 6 & 3 & 4 & 3 & 4 & 3 \\ \hline
$\equilength = 2^{-5}$ & 4 & 3 & 3 & 3 & 3 & 3 & 4 & 3 & 3 & 2 \\ \hline
$\equilength = 2^{-6}$ & 2 & 2 & 3 & 3 & 2 & 2 & 3 & 2 & 3 & 2 \\ \hline
$\equilength = 2^{-7}$ & 2 & 2 & 2 & 2 & 2 & 2 & 1 & 2 & 2 & 2 \\ \hline
\end{tabular}
\caption{Number of required Newton-trace iterations up to convergence to $\eigvalquantum_q$.}\label{tab:Newton_iter}
\end{table}

\pagebreak
\section{Conclusion and Discussion}\label{sec:conclusion}
In this work, we have investigated the spectrum of the negative second order derivative acting on functions on metric graphs. 
Generically, the eigenvalues of such a quantum graph are represented by the solutions of an $n \times n$ nonlinear eigenvalue problem, where $n$ is the number of vertices of the metric graph.
In the special case of equilateral graphs, the NEP further reduces to a linear eigenvalue problem, which can be solved by classical algorithms.
For the numerical solution of the NEP in the general case of non-equilateral graphs, we have developed a procedure to generate estimates of the eigenvalues by approximating the metric graph with equilateral extended graphs.
The eigenvalues of equilateral approximations in turn can be computed as the solutions of the linear eigenvalue problem mentioned above.
These estimates are then applied as initial values in a classical NEP solver, such as a Newton-trace iteration, delivering the solution after a moderate number of iterations, given the approximation is good enough. 
\\
When generating eigenvalue estimates through equilateral approximation, efficient algorithms for solving large scale linear eigenvalue problems on extended graphs are necessary.
We proposed a nested iteration approach to determine initial guesses for an inverse iteration. However, the method can be further improved using, for example, a multigrid method for the solution of the arising systems of linear equations.
\\

Moreover, as part of future work, error estimates for the approximation of quantum graph eigenvalues by equilateral approximations must be addressed. This will be crucial to obtain a measure of accuracy of the estimates. And, as a next step, to derive conditions on the eigenvalue estimates that guarantee a successful application of a Newton iteration.
\\
The numerical experiments in Section \ref{sec:numerical_results} suggest that even fairly rough approximations with $\equilength \approx 10^{-4}$ guarantee the convergence of the Newton-trace iteration to the ``correct'' eigenvalues. However, this might still be very expensive when the scale of the edge lengths of the graph varies greatly. 
\\
Indeed, in a recent work on quantum graphs, Hofmann also studies the estimation of non-equilateral quantum graph eigenvalues by equilateral graphs and proves error estimates in the special case of length preserving approximations, i.e., equilateral approximations with the same total length as the non-equilateral graph of interest \cite{hofmann_2021}. 
However, he does not apply these approximations in the context of NEPs, i.e., as initial guesses for an NEP solver or a related idea.
\\

During this article, we focused on the computation of quantum graph eigenvalues and somewhat neglected the study of the corresponding eigenfunctions. 
Remarkably, it is possible to give an explicit characterization of these eigenfunctions via the solution vectors of the NEP, i.e., the vectors $\eigvec \in \R^n$ with $\bfH(\eigvalquantum) \eigvec = 0$. 
This is due to the fact that the eigenfunctions are completely determined by their values at the vertices\footnote{Indeed, this does not hold for some special non-vertex eigenfunctions, see Remark \ref{re:nonvertex}.
For their treatment, we again refer to \cite{Weller_2023}.}, compare \eqref{al:vertex_eigfunc_boundary}.\\[5mm]

\bibliographystyle{unsrt} 
\bibliography{literature}

\end{document}